\documentclass{article}

\usepackage{amsthm}
\usepackage{amsmath,latexsym,amsfonts,amssymb,graphicx}
\usepackage{algorithm,algpseudocode}
\usepackage{graphicx}[dvipdfmx]
\usepackage[margin=20mm]{geometry}

\newcommand{\R}{\mathbb{R}}

\def\vec#1{\mbox{\boldmath $#1$}}

\newtheorem{theorem}{Theorem}
\newtheorem{lemma}{Lemma}
\newtheorem{remark}{Remark}

\newtheorem{definition}{Definition}

\newtheorem{assumption}{Assumption}

\begin{document}


\title{Consistent estimation with the use of orthogonal projections \\
for a linear regression model with errors in the variables}

\author{Kensuke Aishima
\thanks{Faculty of Computer and Information Sciences,  
  Hosei University, Tokyo 184-8585, Japan
  (\texttt{aishima@hosei.ac.jp}).}
}

\maketitle

\begin{abstract}
In this paper, we construct an estimator 
of an errors-in-variables linear regression model.
The regression model leads to a constrained total least squares problems
with row and column constraints.
Although this problem can be numerically solved,
it is unknown whether the solution has consistency in the statistical sense.
The proposed estimator can be constructed 
by the use of orthogonal projections and their properties,
its strong consistency is naturally proved.
Moreover, our asymptotic analysis proves the strong consistency
of the total least squares solution of
the problem with row and column constraints.

\end{abstract}




\section{Introduction}
Linear parametric statistical models arise in
many scientific and engineering problems,
in particular a data fitting problem.
With such parametric models in mind, 
we focus on an overdetermined system 
$AX \approx B$ for given $A\in \R^{m\times n}$ and $B\in \R^{m\times \ell}$, which are contaminated with noise,
where $X\in \R^{n\times \ell}$ is an unknown matrix to be estimated.
The total least squares method (TLS)~\cite{GV1980,GV2013} is a powerful technique
to find $X$ in a reasonable mathematical formulation. 
The TLS is computed by the the singular value decomposition (SVD)
due to the Eckart-Young-Mirsky theorem.
The TLS is
interpreted as an estimator in the statistical sense,
and the estimator has strong consistency
under some assumptions on the random noise~\cite{Gleser1981}.

The TLS has many variants
adopted to the problem formulations.
As a typical example, there is a rank constrained 
TLS problem, which can be solved by
the SVD, and some probabilistic analysis is presented in~\cite{Park2011}.
In this paper, we are concerned with other
constrained TLS problems. Suppose that,
for the matrix $A$, the first $k$ columns are exactly known
without noise. In addition, the first $j$ rows of $A$ and $B$
are exactly known without noise.
Such a problem can be formulated as a
constrained TLS problem, which can be solved
by the SVD and the QR decomposition~\cite{Demmel1987,Huffel1991,wei1998}.
The discussion of the column constraints
is detailed in \cite{GHS1987,Huffel1989}.
Recently, the problem with the row constraints 
has attracted much attention in many application areas,
and thus numerical method is revived with
mathematical analysis for the solvability conditions~\cite{LJY2022}.
With the above pioneering study, numerical and applied research 
on the TLS problem and its various variants is currently underway.

In this paper, we aim to perform asymptotic statistical analysis 
for large sample size, namely consistency analysis.
The strong consistency of the original total least squares solution
is proved by Gleser~\cite{Gleser1981}.
In addition, there exist some theoretical results for the consistency analysis.
Among them, the consistency proved by Gleser~\cite{Gleser1981} is easily extended to
the column constraints TLS problem as in~\cite[\S 3]{Huffel1991}.
However, the consistency of the solution of more general cases
involving the row constraints was stated as unproved,
and it appears to have remained an open problem
for a long time in subsequent researches.
Recently, for the row constraints TLS problem,
the strong consistency of the numerical solution has been proved
under reasonable assumptions~\cite{Aishima2022}.
However, the case of both column and row constraints
is not discussed in any existing study,
as far as the author knows.

With this research background, we construct an estimator with strong consistency for problems that are constrained in both columns and rows. Technically, the contribution is that the algorithm design focuses on the properties of orthogonal projections, rather than on the usual optimization methods, which naturally proves strong consistency. This strategy leads to a straightforward proof that the TLS estimator for the constrained problem has strong consistency.
This research is not a proposal for a probabilistic method for efficiently solving problems in numerical linear algebra. It is the result of proving strong consistency in a statistical sense using numerical techniques in numerical linear algebra, and the contribution of this paper is the indication of such a research direction.

This paper is organized as follows.
The next section is devoted to a description of the previous study
to clarify the motivation for this study.
Section~\ref{sec:aim} describes our target regression model
to construct an estimator with consistency analysis. 
In Section~\ref{sec:mainresult}, we propose an estimator
with consistency analysis focusing on the orthogonal projections. 
In Section~\ref{sec:mainresult2},
we prove strong consistency of the TLS estimator,
with emphasis on the orthogonal projections.
Finally, Section~\ref{sec:conclusion} gives the conclusion.

\section{Previous study}\label{sec:prev}
This section is devoted to a description of the previous study to formulate our target problems.
Throughout the paper, $I$ is an identity matrix of an appropriate size,
and $0_{p\times q}$ is a $p\times q$ zero matrix.

\subsection{Total least squares solution and its strong consistency for a regression model}
Given $A\in \R^{m\times n}$ and ${B} \in \R^{m\times \ell}$,
the total least squares (TLS) is to find ${X}_{\rm tls} \in \R^{n\times \ell}$
for the following optimization problem:
\begin{align}\label{eq:tls}
\min_{{X}_{\rm tls}, \Delta A, \Delta {B}} {\|\Delta A\|_{\rm F}}^{2}+{\|\Delta {B}\|_{\rm F}}^{2}
\quad {\rm such \ that}\  (A+\Delta A){X}_{\rm tls}={B}+\Delta {B}.
\end{align}

The solution vector ${X}_{\rm tls}$ can be obtained
by the singular value decomposition (SVD)
under reasonable assumptions~\cite{GV1980}.
It is due to the Eckart-Young-Mirsky theorem.
Here we describe how to compute ${X}_{\rm tls}$ 
with the use of the eigenvalue decomposition
to consider a statistical model.

First, we define $C$, $\Delta C$, and $Y$ such that
\begin{align}\nonumber
C=
\begin{bmatrix}
A & B 
\end{bmatrix}\in \R^{m\times (n+\ell)},
\quad
\Delta{C}=
\begin{bmatrix}
\Delta{A} & \Delta{B} 
\end{bmatrix}\in \R^{m\times (n+\ell)},
\quad
Y=
\begin{bmatrix}
X_{\rm tls} \\ -I
\end{bmatrix}\in \R^{(n+\ell)\times \ell}.
\end{align}
Then we have
\begin{align}\nonumber
(C+\Delta{C})Y=0_{m\times \ell}
\end{align}
from the constraint in \eqref{eq:tls}.
For solving the optimization problem \eqref{eq:tls},
let 
\begin{align}\nonumber
F=C^{\top}C\in \R^{(n+\ell)\times (n+\ell)}.
\end{align}
In addition, for $i=1,\ldots ,n+\ell$,
let $\lambda_{i}$ and $\vec{z}_{i}$ denote the eigenpairs, i.e., 
\begin{align}\nonumber
F\vec{z}_{i}=\lambda_{i}\vec{z}_{i}\ (i=1,\ldots ,n+\ell),\quad
(0\le )\lambda_{1}\le \cdots \le \lambda_{n+\ell}.
\end{align}
Moreover, for a part of the eigenpairs, 
let
\begin{align}\nonumber
\Lambda_{\ell}={\rm diag}(\lambda_{1}, \ldots , \lambda_{\ell}),\quad
Z_{\ell}=[\vec{z}_{1},\ldots ,\vec{z}_{\ell}].
\end{align}
The eigenvector matrix $Z_{\ell}$ is used to solve \eqref{eq:tls}.
Let $Z_{\ell}$ be divided into
\begin{align}\nonumber
Z_{\ell}=:
\begin{bmatrix}
Z_{\ell,{\rm upper}}
\\
Z_{\ell,{\rm lower}}
\end{bmatrix},\ 
Z_{\ell,{\rm upper}}\in \R^{n\times \ell},\
Z_{\ell,{\rm lower}}\in \R^{\ell \times \ell}.
\end{align}
Then we have
\begin{align}\nonumber
X_{\rm tls}=-Z_{\ell,{\rm upper}}{Z_{\ell,{\rm lower}}}^{-1},
\end{align}
where the theoretical background is described in~\cite{Gleser1981,GV1980} and~\cite[\S 6.3]{GV2013}.

Here we consider a statistical problem setting
relevant to \eqref{eq:tls}
to analyze the property of $X_{\rm tls}$ as the estimator in the statistical sense.
To this end, let $\bar{A}\in \R^{m\times n}$ and $\bar{{B}}\in \R^{m\times \ell}$
denote parameters depending on $m$.
More specifically,
\begin{align*}
\bar{A}:=[\bar{\vec{a}}_{1},\ldots ,\bar{\vec{a}}_{m}]^{\top},\quad
\bar{{B}}:=[\bar{\vec{b}}_{1},\ldots ,\bar{\vec{b}}_{m}]^{\top},
\end{align*}
where $\bar{\vec{a}}_{i}\in \R^{n}$ and $\bar{\vec{b}}_{i}\in \R^{\ell}$
for $i=1,\ldots ,m$.
It then follows that
${\bar{\vec{a}}_{i}}{}^{\top}{X}=\bar{\vec{b}}_{i}{}^{\top}\ (i=1,\ldots m)$
for some ${X}\in \R^{n\times \ell}$ in the linear regression model. 
The errors-in-variables linear regression model
is formulated as
\begin{align}\label{eq:tlsregression}
\bar{A}{X}=\bar{{B}},\quad 
A=\bar{A}+E_{A},\quad 
{B}=\bar{{B}}+{E}_{{B}},
\end{align}
where all the elements of $E_{A}$ and ${E}_{{B}}$
are random variables, and only $A$ and ${B}$ can be observed.
Hence, letting
\begin{align}\label{eq:E}
E:=
\begin{bmatrix}
 E_{A} &  {E}_{{B}} 
\end{bmatrix},
\end{align}
we assume the following.
\begin{assumption}\label{as:E}
The rows of $E$ are i.i.d. with common mean row vector $\vec{0}$ and common covariance matrix $\sigma^2 I$. 
\end{assumption}

\begin{remark}
If the common covariance matrix is a general form $\Sigma$,
the Cholesky factorization $\Sigma=LL^{\top}$
can reduce this case to the simple problem under Assumption~\ref{as:E}
as in \cite[\S 5]{Gleser1981}.
In addition, as in~\cite{Huffel1989},
the general covariance matrix can be formulated
by the generalized eigenvalue problem,
though the formulation is
mathematically equivalent to the remedy in \cite[\S 5]{Gleser1981}.
Since our purpose is mathematical analysis, 
Assumption~\ref{as:E} covers the general $\Sigma$ with the above remedy. 
Therefore, in the following, we perform consistency analysis under Assumption~\ref{as:E},
where the standard deviation $\sigma$ is unknown.
\end{remark}

In addition, assume that
$\lim_{m\to \infty}m^{-1}\bar{A}^{\top}\bar{A}$ exists, and it is positive definite,
related to the uniqueness of ${X}$ in the asymptotic regime as $m \to \infty$.
Below we rewrite the assumption
for extending the regression model in the next section. 
From such a perspective, let
\begin{align}\label{eq:barC}
\bar{C}:=
\begin{bmatrix}
\bar{A} & \bar{{B}} 
\end{bmatrix}=
\begin{bmatrix}
\bar{A} & \bar{A}{X} 
\end{bmatrix}.
\end{align}
From easy calculations, we have
\begin{align*}
m^{-1}\bar{C}^{\top}\bar{C}=
\begin{bmatrix}
\ m^{-1}\bar{A}^{\top}\bar{A} & m^{-1}\bar{A}^{\top}\bar{A}{X} \\
\ m^{-1}{X}^{\top}\bar{A}^{\top}\bar{A} & m^{-1}{X}^{\top}\bar{A}^{\top}\bar{A}{X}
\end{bmatrix}.
\end{align*}
The assumption that $\lim_{m\to \infty}m^{-1}\bar{A}^{\top}\bar{A}$ is
positive definite is mathematically the same as the next assumption.
\begin{assumption}\label{as:CTC}
$\lim_{m\to \infty}m^{-1}\bar{C}^{\top}\bar{C}=:\bar{S}$ exists, and ${\rm rank}(\bar{S})=n$. 
\end{assumption}

Then we have the next theorem
that states the strong consistency of the estimation ${X}_{\rm tls}$ in \eqref{eq:tls}.
\begin{theorem}[Strong consistency~{\cite[Lemma~3.3]{Gleser1981}}]\label{thm:Gleser}
Under Assumptions~\ref{as:E} and~\ref{as:CTC},
we have
\begin{align*}
\lim_{m\to \infty}{X}_{\rm tls}={X} \ \text{with probability one},
\end{align*}
where ${X}_{\rm tls}$ is the solution vector of \eqref{eq:tls}.
\end{theorem}

In the following, we consider extensions of the TLS problem and the corresponding consistency analysis.

\subsection{Constrained total least squares problems and their consistency analysis}
Here, let $A$ be divided into
\begin{align}\nonumber
A=\begin{bmatrix}
{A}_{1} & {A}_{2}
\end{bmatrix},\quad
{A}_{1}\in \R^{m\times k},\quad {A}_{2}\in \R^{m\times (n-k)}.
\end{align}
We consider the following optimization problem:
\begin{align}\label{eq:ctls}
\min_{{X}_{\rm ctls}, \Delta A, \Delta {B}} {\|\Delta A\|_{\rm F}}^{2}+{\|\Delta {B}\|_{\rm F}}^{2}
\quad {\rm such \ that}\  \begin{bmatrix}
{A}_{1} & {A}_{2}+\Delta A 
\end{bmatrix}{X}_{\rm ctls}=
{{B}}+\Delta {B}.
\end{align}

The above problem is the so-called constrained
total least squares problem,
where $A_{1}$ comprising the first $k$ columns vectors
is fixed, different from~\eqref{eq:tls}. 
The constrained TLS problem can be reduced to
the TLS problem with the use of the QR decomposition of $A_{1}$:
\begin{align}\nonumber
A_{1}=QR,\quad
Q=
\begin{bmatrix}
{Q}_{1} & {Q}_{2}
\end{bmatrix},\quad
R=
\begin{bmatrix}
{R}_{1} \\
0_{(m-k)\times k}
\end{bmatrix},\quad
Q_{1}\in \R^{m\times k},\quad
Q_{2}\in \R^{m\times (m-k)},\quad
R_{1} \in \R^{k\times k}
\end{align}
as the preconditioning,
where $R_{1}$ is an upper triangular matrix.
The above definitions mean
\begin{align}\nonumber
A_{1}=Q_{1}R_{1}.
\end{align}
Applying the orthogonal transformation by the above $Q$, we have
\begin{align}\nonumber
Q^{\top}
\begin{bmatrix}
{A}_{1} & {A}_{2} & B
\end{bmatrix}
=
\begin{bmatrix}
{R}_{1} & {Q_{1}}^{\top}{A}_{2} & {Q_{1}}^{\top}{B} \\
0_{(m-k)\times k}       & {Q_{2}}^{\top}{A}_{2} & {Q_{2}}^{\top}{B} 
\end{bmatrix}.
\end{align}
Briefly speaking, the problem is reduced to 
the small size TLS problem associated with
the lower right submatrix
$[{Q_{2}}^{\top}{A}_{2} \ \ {Q_{2}}^{\top}{B}]$;
see~\cite{Demmel1987,GHS1987,Huffel1989} for the details.

The corresponding regression model is
\begin{align}\label{eq:crmodel}
\begin{bmatrix}
\bar{A}_1 & \bar{A}_2 
\end{bmatrix}
{X}=\bar{{B}},\quad 
\begin{bmatrix}
{A}_1 & {A}_2 
\end{bmatrix}
=\begin{bmatrix}
\bar{A}_1 & \bar{A}_2+E_{A_2} 
\end{bmatrix}
,\quad 
{B}=\bar{{B}}+{E}_{{B}},
\end{align}
where all the elements of $E_{A}$ and ${E}_{{B}}$
are random variables, and only $A$ and ${B}$ can be observed.
Hence, letting
\begin{align}\label{eq:E2}
E:=
\begin{bmatrix}
E_{A_2} & {E}_{{B}} 
\end{bmatrix},
\end{align}
we have the next theorem
that states the strong consistency of the estimation ${X}_{\rm ctls}$ in \eqref{eq:ctls}.
\begin{theorem}[Strong consistency~{\cite[\S 3]{Huffel1991}}]\label{thm:Huffel}
Under Assumptions~\ref{as:E} and~\ref{as:CTC},
we have
\begin{align*}
\lim_{m\to \infty}{X}_{\rm ctls}={X} \ \text{with probability one},
\end{align*}
where ${X}_{\rm ctls}$ is the solution vector of \eqref{eq:ctls}.
\end{theorem}

Next, we consider another constraint in the TLS problem.
Here, let $A$ be divided into
\begin{align}\nonumber
A=\begin{bmatrix}
{A}_{1} \\
{A}_{2}
\end{bmatrix},\quad
{A}_{1}\in \R^{j\times n},\quad {A}_{2}\in \R^{(m-j)\times n}.
\end{align}
Similarly,
\begin{align}\nonumber
B=\begin{bmatrix}
{B}_{1} \\
{B}_{2}
\end{bmatrix},\quad
{B}_{1}\in \R^{j\times \ell},\quad {B}_{2}\in \R^{(m-j)\times \ell}.
\end{align}

We consider the following optimization problem:
\begin{align}\label{eq:ctls2}
\min_{{X}_{\rm ctls}, \Delta A, \Delta {B}} {\|\Delta A\|_{\rm F}}^{2}+{\|\Delta {B}\|_{\rm F}}^{2}
\quad {\rm such \ that}\  \begin{bmatrix}
{A}_{1} \\
 {A}_{2}+\Delta A 
\end{bmatrix}{X}_{\rm ctls}=
\begin{bmatrix}
B_{1}
\\
B_2+\Delta {B}_{2}
\end{bmatrix},
\end{align}
where the first $j$ rows of $A$ and $B$ are fixed.
Analogously to~\eqref{eq:ctls}, the constrained TLS problem~\eqref{eq:ctls2} can be solved by the orthogonal
transformations; see~\cite{LJY2022} for the details.

Here we consider the corresponding regression model
described as follows:
\begin{align}\nonumber
\begin{bmatrix}
\bar{A}_1 \\
\bar{A}_2 
\end{bmatrix}
{X}=\begin{bmatrix}
\bar{B}_1 \\
\bar{B}_2 
\end{bmatrix}
,\quad 
A=\begin{bmatrix}
\bar{A}_1 \\ 
\bar{A}_2+E_{A_2} 
\end{bmatrix}
,\quad 
{B}=\begin{bmatrix}
\bar{B}_1 \\
\bar{B}_2+E_{B_{2}} 
\end{bmatrix},
\end{align}
where all the elements of $E_{A_2}$ and ${E}_{{B}_2}$
are random variables, and only $A$ and ${B}$ can be observed.
Related to this formulation, noting that $\bar{A}_{1}=A_{1}$,
we assume ${\rm rank}(A_{1})=j<n$.
If not, we select independent rows in ${A}_{1}$ to recover $X$.
Thus, ${\rm rank}(A_{1})=j<n$ is not restrictive.
For the random matrices, let
\begin{align}\label{eq:E3}
E:=
\begin{bmatrix}
E_{A_2} & {E}_{{B}_2} 
\end{bmatrix}.
\end{align}
In addition, similarly to~\eqref{eq:barC}, let
\begin{align*}
\bar{C}:=
\begin{bmatrix}
\bar{A}_2 & \bar{B}_2 
\end{bmatrix}.
\end{align*}
Moreover, define $P\in \R^{(n+\ell)\times (n+\ell-j)}$ such that
\begin{align}\nonumber
P^{\top}P=I,
\quad
\begin{bmatrix}
A_1 & {B}_1 
\end{bmatrix}P=0_{j\times (n+\ell-j)}.
\end{align}
The matrix $P$ corresponds to the orthonormal basis vectors
of the orthogonal complementary subspace of $[A_1 \ \ {B}_1]$.
We extend Assumption~\ref{as:CTC} to the following form.
\begin{assumption}\label{as:PCCP}
$\lim_{m\to \infty}m^{-1}{P}^{\top}\bar{C}^{\top}\bar{C}{P}=:\bar{T}$ exists, and ${\rm rank}(\bar{T})=n-j$. 
\end{assumption}

Since the case of $j=0$ corresponds to $P=I$,
Assumption~\ref{as:PCCP} for $j=0$ is reduced to 
Assumption~\ref{as:CTC}.
The next theorem states the strong consistency of the estimation ${X}_{\rm ctls}$ in \eqref{eq:ctls2}.
\begin{theorem}[Strong consistency~{\cite[Theorem~2]{Aishima2022}}]\label{thm:Aishima}
For $\ell=1$ and a fixed $j$, 
under Assumptions~\ref{as:E} and~\ref{as:PCCP},
we have
\begin{align*}
\lim_{m\to \infty}{X}_{\rm ctls}={X} \ \text{with probability one},
\end{align*}
where ${X}_{\rm ctls}$ is the solution vector of \eqref{eq:ctls2}.
\end{theorem}

For the general $\ell \ge 1$, we prove the strong consistency in Section~\ref{sec:mainresult2}
in a more extended case.

\subsection{Total least squares problems with row and column constraints}
\label{sec:CTLSrc}
Here we consider a general optimization problem
covering \eqref{eq:ctls} and \eqref{eq:ctls2},
described as follows:
\begin{align}\label{eq:ctlscr}
\min_{{X}_{\rm ctls}, \Delta A, \Delta {B}} {\|\Delta A\|_{\rm F}}^{2}+{\|\Delta {B}\|_{\rm F}}^{2}
\quad {\rm such \ that}\  \begin{bmatrix}
{A}_{11} & {A}_{12} \\
{A}_{21} & {A}_{22}+\Delta A 
\end{bmatrix}{X}_{\rm ctls}=
\begin{bmatrix}
{{B}}_{1} \\
{{B}}_{2}+\Delta {B} \\
\end{bmatrix},
\end{align}
where 
\begin{align}\nonumber
A_{11}\in \R^{j\times k},\quad 
A_{12}\in \R^{j\times (n-k)},\quad
A_{21}\in \R^{(m-j)\times k},\quad
A_{22}\in \R^{(m-j)\times (n-k)},\quad
B_{1}\in \R^{j\times \ell},\quad
B_{2}\in \R^{(m-j)\times \ell}.
\end{align}
This problem can be solved by the preconditioning
with the aid of the Gaussian elimination, i.e., the LU factorization.
Below we explain the preconditioning.
First, let us consider the case in which
$A_{11}$ is square and nonsingular.
Then noting
\begin{align}\nonumber
\begin{bmatrix}
I & 0_{j\times (n-k)} \\
-A_{21}{A_{11}}^{-1} & I
\end{bmatrix}
\begin{bmatrix}
{A}_{11} & {A}_{12} \\
{A}_{21} & {A}_{22}+\Delta A 
\end{bmatrix}
=
\begin{bmatrix}
{A}_{11} & {A}_{12} \\
0_{(m-j)\times k} & {A}_{22}-{A}_{21}{{A}_{11}}^{-1}{A}_{12}+\Delta A 
\end{bmatrix},
\end{align}
we have
\begin{align}\label{eq:ctlsreduced}
\begin{bmatrix}
{A}_{11} & {A}_{12} \\
0_{(m-j)\times k} & {A}_{22}-{A}_{21}{{A}_{11}}^{-1}{A}_{12}+\Delta A \\
\end{bmatrix}{X}_{\rm ctls}=
\begin{bmatrix}
{{B}}_{1} \\
{{B}}_{2}-{A}_{21}{{A}_{11}}^{-1}{{B}}_{1}+\Delta {B} \\
\end{bmatrix}.
\end{align}
Thus the problem is reduced to the TLS problem
for the lower right $(m-j)\times (n-k)$ submatrix
in the left-hand sides.
In other words, letting $\widetilde{X}_{\rm ctls}\in \R^{(n-k)\times \ell}$ denote
the lower submatrix of $X_{\rm ctls}$, we have
\begin{align}\nonumber
({A}_{22}-{A}_{21}{{A}_{11}}^{-1}{A}_{12}+\Delta A)\widetilde{X}_{\rm ctls}=
{{B}}_{2}-{A}_{21}{{A}_{11}}^{-1}{{B}}_{1}+\Delta {B},
\end{align}
regarded as the TLS problem.

If $A_{11}$ is not square, we consider the following computations,
including the rank deficient case.
Let $r:={\rm rank}(A_{11})$.
Without loss of generality, since the singular value decomposition (SVD)
is possible for any matrix, we compute the SVD of $A_{11}$:
\begin{align}\nonumber
A_{11}=U\Sigma V^{\top},
\end{align}
where $U\in \R^{j\times j}$ and $V\in \R^{k\times k}$ are orthogonal matrices and
\begin{align}\nonumber
\Sigma=
\begin{bmatrix}
\Sigma_{r} & 0_{r \times (k-r)} \\
0_{(j-r) \times r} &  0_{(j-r)\times (k-r)}
\end{bmatrix},\quad
\Sigma_{r}={\rm diag}(\sigma_{1},\ldots ,\sigma_{r}),\quad \sigma_{1}\ge \cdots \ge \sigma_{r}>0.
\end{align}
Using the SVD, we have
\begin{align}\nonumber
\begin{bmatrix}
U^{\top} & 0_{j\times (m-j)}  \\
0_{(m-j)\times j} &  I_{m-j}
\end{bmatrix}
\begin{bmatrix}
{A}_{11} & {A}_{12} \\
{A}_{21} & {A}_{22}+\Delta A
\end{bmatrix}
\begin{bmatrix}
V & 0_{k\times (n-k)}  \\
0_{(n-k)\times k} &  I_{n-k}
\end{bmatrix}
=
\begin{bmatrix}
\Sigma & U^{\top}A_{12}  \\
A_{21}V &  A_{22}+\Delta A
\end{bmatrix}.
\end{align}
In the right-hand side,
since the $r\times r$ leading principal submatrix is nonsingular (and diagonal),
the Gaussian elimination can be applied as the preconditioning.
As a result, without loss of generality,
it is enough to consider the matrix $A$
such that the upper left submatrix is the zero matrix.
Thus we replace the constraint in \eqref{eq:ctlscr} with
\begin{align}\label{eq:Xctlsprob}
\begin{bmatrix}
0_{j\times k} & {A}_{12} \\
{A}_{21} & {A}_{22}+\Delta A \\
\end{bmatrix} {X}_{\rm ctls}=
\begin{bmatrix}
{{B}}_{1} \\
{{B}}_{2}+\Delta {B} \\
\end{bmatrix}.
\end{align}
This ${X}_{\rm ctls}$ for the minimum perturbation 
in terms of ${\|\Delta A\|_{\rm F}}^{2}+{\|\Delta B\|_{\rm F}}^{2}$
can be obtained by the QR decomposition and the SVD; see~\cite{Demmel1987} for details.

One of the contributions of this study 
is to prove strong consistency
of $X_{\rm ctls}$ in the statistical sense
under reasonable assumptions,
as presented in Section~\ref{sec:mainresult2}.

\section{Problem setting and the aim of this study}\label{sec:aim}
With the research background in the previous section, 
we formulate the target regression model
to construct an estimator with consistency analysis. 

Basically, we consider a regression model
associated with the constrained TLS problem of \eqref{eq:ctlscr}
in Section~\ref{sec:CTLSrc}.
Hence, let
\begin{align}\label{eq:defbarA}
\bar{A}=
\begin{bmatrix}
\bar{A}_{11} & \bar{A}_{12} \\
\bar{A}_{21} & \bar{A}_{22} \\
\end{bmatrix}\in \R^{m\times n},\
\bar{A}_{11}\in \R^{j\times k},\
\bar{A}_{12}\in \R^{j\times (n-k)},\
 \bar{A}_{21}\in \R^{(m-j)\times k},\
\bar{A}_{22}\in \R^{(m-j)\times (n-k)}.
\end{align}
In addition, let
\begin{align}\label{eq:defbarB}
\bar{{B}}=
\begin{bmatrix}
\bar{{B}}_{1} \\
\bar{{B}}_{2} \\
\end{bmatrix},\
\bar{{B}}_{1}\in \R^{j\times \ell},\
\bar{{B}}_{2}\in \R^{(m-j)\times \ell}.
\end{align}
Using the definitions above, we consider the following  regression model:
\begin{align}\label{eq:defbarAXbarB}
{X}\in \R^{n\times \ell},\quad \bar{A}{X}=\bar{{B}},
\quad
A=
\begin{bmatrix}
\bar{A}_{11} & \bar{A}_{12} \\
\bar{A}_{21} & \bar{A}_{22}+E_{A_{22}} \\
\end{bmatrix},\
{{B}}=
\begin{bmatrix}
\bar{{B}}_{1} \\
\bar{{B}}_{2}+{E}_{{B}_2} \\
\end{bmatrix},
\end{align}
where $E_{A_{22}}$ and $E_{B_{2}}$ are random matrices.
Without loss of generality,
from~\eqref{eq:ctlscr} and \eqref{eq:defbarAXbarB}, we assume
\begin{align}\label{eq:A11A12}
{\rm rank}\left(
\begin{bmatrix}
\bar{A}_{11} & \bar{A}_{12} 
\end{bmatrix}\right)
=
{\rm rank}\left(
\begin{bmatrix}
{A}_{11} & {A}_{12} 
\end{bmatrix}\right)
=j<n.
\end{align}
If not, we select independent rows in $[\bar{A}_{12}\ \ \bar{A}_{22}]$ to recover $X$.
Thus, the rank requirement above is not restrictive.
For the random matrices, let
\begin{align*}
E=
\begin{bmatrix}
E_{A_{22}} & E_{B_{2}}
\end{bmatrix}.
\end{align*}
In addition, let
\begin{align*}
\bar{C}=
\begin{bmatrix}
\bar{A} & \bar{B}
\end{bmatrix}
\in \R^{m\times (n+\ell)},\quad
{Y}=
\begin{bmatrix}
X \\
-I
\end{bmatrix}
\in \R^{(n+\ell)\times \ell}.
\end{align*}
From $\bar{A}X=\bar{B}$ in \eqref{eq:defbarAXbarB},
we can see that
\begin{align*}
\bar{C}{Y}=0_{m\times \ell}.
\end{align*}
Moreover, for the consistency analysis,
let $\bar{C}$ be divided into 
\begin{align}\label{eq:barC2}
\bar{C}=
\begin{bmatrix}
\bar{C}_{11} & \bar{C}_{12} \\
\bar{C}_{21} & \bar{C}_{22} 
\end{bmatrix},
\end{align}
where
\begin{align}\label{eq:defbarC}
\bar{C}_{11}=\bar{A}_{11}\in \R^{j\times k},\
\bar{C}_{12}=
\begin{bmatrix}
\bar{A}_{12} & \bar{{B}}_{1}
\end{bmatrix}
\in \R^{j\times (n-k+\ell)},\
 \bar{C}_{21}= \bar{A}_{21}\in \R^{(m-j)\times k},\
\bar{C}_{22}=
\begin{bmatrix}
\bar{A}_{22} & \bar{{B}}_{2}
\end{bmatrix}
\in \R^{(m-j)\times (n-k+\ell)}.
\end{align}
Similarly, define
\begin{align}\label{eq:C}
{C}=
\begin{bmatrix}
{C}_{11} & {C}_{12} \\
{C}_{21} & {C}_{22} 
\end{bmatrix}
\end{align}
such that
\begin{align}\label{eq:defC}
{C}_{11}={A}_{11}\in \R^{j\times k},\
{C}_{12}=
\begin{bmatrix}
{A}_{12} & {B}_{1} 
\end{bmatrix}
\in \R^{j\times (n-k+\ell)},\
{C}_{21}= {A}_{21}\in \R^{(m-j)\times k},\
{C}_{22}=
\begin{bmatrix}
{A}_{22} & {B}_{2} 
\end{bmatrix}
\in \R^{(m-j)\times (n-k+\ell)},
\end{align}
where
\begin{align}\label{eq:CE}
C_{11}=\bar{C}_{11},\quad
C_{12}=\bar{C}_{12},\quad
C_{21}=\bar{C}_{21},\quad
C_{22}=\bar{C}_{22}+E
\end{align}
from the constraints in the problem setting~\eqref{eq:defbarAXbarB}.

For the statistical asymptotic analysis, 
we define the statistical terms exactly.
In general, the strong convergence of random variables
is defined as follows.

\begin{definition}[Strong convergence]\label{def:sc}
The sequence of random variables $\mathcal{X}^{(m)}$
for $m=0,1,\ldots$ converges to $\mathcal{X}$ with probability one
if and only if
\begin{align*}
\mathcal{P}(\{ \omega \in \Omega \mid \lim_{m\to \infty}\mathcal{X}^{(m)}(\omega)=\mathcal{X}(\omega) \})=1,
\end{align*}
where $\mathcal{P}$ is a probability measure on a sample space $\Omega$.
\end{definition}

In this paper, we consider the case that
$\mathcal{X}$ in Definition~\ref{def:sc}
is not a random variable but a constant
in a regression model.
In the following, the random variable argument $\omega$ is omitted for simplicity.
Let $\widehat{\mathcal{X}}^{(m)}$ denote 
an estimator of $\mathcal{X}$
with the use of a sample of size $m$.
If $\widehat{\mathcal{X}}^{(m)}$
converges to $\mathcal{X}$ with probability one,
the estimator $\widehat{\mathcal{X}}^{(m)}$
is said to be strongly consistent.
Then $\widehat{\mathcal{X}}^{(m)}$ has the strong consistency.

In the following, we discuss the strong consistency
of estimators of $X$ in a regression model.
For the consistency analysis, we use the next lemma,
essentially proved in~\cite[Lemma~3.1]{Gleser1981}.
The lemma is also due to the Kolmogolov strong law of large numbers.

\begin{lemma}\label{lem:kolmogolov}
Let $\{\alpha_{m}\}_{m=1}^{\infty}$ denote a sequence of real numbers
such that $m^{-1}\sum_{i=1}^{m}{\alpha_{i}}^2$ are bounded, i.e.,
$m^{-1}\sum_{i=1}^{m}{\alpha_{i}}^2 < \infty$ for all $m$.
In addition,
let $\{\epsilon^{(m)}\}_{m=1}^{\infty}$ denote a sequence of 
independent and identically distributed (i.i.d.) random variables
with $\mathbb{E}(\epsilon^{(m)})=0$ for every $m$,
where $\mathbb{E}$ means the expectation of random variables.
Moreover, suppose that, for all $m$, 
the variances of $\epsilon^{(m)}$,
i.e., $\mathbb{E}({\epsilon^{(m)}}^{2})$
are bounded.
Then 
\begin{align}\label{eq:keyslln}
\lim_{m\to \infty}m^{-1}\sum_{i=1}^{m}\alpha_{i}\epsilon^{(i)}=0 \ \text{with probability one}.
\end{align}
\end{lemma}
\begin{proof}
First of all, let
\begin{align*}
\beta_{0}=0,\quad \beta_{m}=\sum_{i=1}^{m}{\alpha_{i}}^{2}\quad (m=1,2,\ldots).
\end{align*}
Then, from the assumption, obviously 
\begin{align}\label{eq:beta}
m^{-1}\beta_{m}<\infty \quad (m=1,2,\ldots).
\end{align}
In addition, noting Abel's partial summation formula,
we have
\begin{align*}
\sum_{i=1}^{m}i^{-2}{\alpha_{i}}^{2}=
\sum_{i=1}^{m}i^{-2}(\beta_{i}-\beta_{i-1})=
\sum_{i=1}^{m}i^{-2}\beta_{i}-\sum_{i=0}^{m-1}(i+1)^{-2}\beta_{i}=
m^{-2}\beta_{m}-\sum_{i=1}^{m-1}(i^{-2}-(i+1)^{-2})\beta_{i}.
\end{align*}
The second term on the right-hand side of the above equation is
$\sum_{i=1}^{m-1}(i^{-1}(i+1)^{-2}(2i+1))i^{-1}\beta_{i}$.
From \eqref{eq:beta} and 
$\sum_{i=1}^{m-1}i^{-1}(i+1)^{-2}(2i+1)<\infty$,
we have $\sum_{i=1}^{m}i^{-2}{\alpha_{i}}^{2}<\infty$.
Thus, from the assumption that $\mathbb{E}(\epsilon^{(i)}{}^{2})<\infty \ (i=1,2,\ldots )$, 
\begin{align*}
\sum_{i=1}^{m}i^{-2}\mathbb{E}((\alpha_{i}\epsilon^{(i)})^{2})=
\sum_{i=1}^{m}i^{-2}{\alpha_{i}}^{2}\mathbb{E}(\epsilon^{(i)}{}^{2})<\infty.
\end{align*}
This boundedness of $\sum_{i=1}^{m}i^{-2}\mathbb{E}((\alpha_{i}\epsilon^{(i)})^{2})$
leads to \eqref{eq:keyslln}
by \cite[eq. (7) in Corollary to Theorem~5.4.1]{Chung2001}.
\end{proof}

With the above preparation for consistency analysis, 
the main findings of this study are presented in the next two sections.

\section{The proposed estimator using the orthogonal projection}\label{sec:mainresult}
In this section, we present a new method for estimating $X$ and prove its strong consistency, 
with emphasis on orthogonal projections.
Before discussing the specific design of the estimation method, 
we discuss the asymptotic properties of the regression model~\eqref{eq:defbarAXbarB} as $m\to \infty$.

\subsection{Problem formulation in terms of deterministic asymptotic analysis}
Here we have a discussion of~\eqref{eq:defbarAXbarB} in the asymptotic regime as $m\to \infty$.
First, let
\begin{align*}
\bar{C}_{\rm upper}=
\begin{bmatrix}
\bar{C}_{11} & \bar{C}_{12} 
\end{bmatrix},
\quad
\bar{C}_{\rm lower}=
\begin{bmatrix}
\bar{C}_{21} & \bar{C}_{22} 
\end{bmatrix},
\quad
{C}_{\rm lower}=
\begin{bmatrix}
{C}_{21} & {C}_{22} 
\end{bmatrix}
\end{align*}
for \eqref{eq:barC2} and \eqref{eq:C}.
From~\eqref{eq:A11A12} and the definitions of $\bar{C}$, we have
\begin{align}\nonumber
{\rm rank}(\bar{C}_{\rm upper})=j.
\end{align}
In other words, $\bar{C}_{\rm upper}$ has the full rank of the rows.
In addition, define $P$ such that
\begin{align}\label{eq:defP}
P\in \R^{(n+\ell)\times (n+\ell-j)},
\quad
P^{\top}P=I_{n+\ell-j},
\quad
\bar{C}_{\rm upper}P=0_{j\times (n+\ell-j)}.
\end{align}
Note that the columns of $P$ are
orthonormal basis vectors of the null space of $\bar{C}_{\rm upper}$.
Moreover, letting
\begin{align}\label{eq:barF}
\bar{F}:={\bar{C}_{\rm lower}}{}^{\top}\bar{C}_{\rm lower}=
\begin{bmatrix}
{\bar{C}_{21}}{}^{\top}\bar{C}_{21} & {\bar{C}_{21}}{}^{\top}\bar{C}_{22} \\
{\bar{C}_{22}}{}^{\top}\bar{C}_{21} & {\bar{C}_{22}}{}^{\top}\bar{C}_{22}  
\end{bmatrix},
\end{align}
we assume the following.

\begin{assumption}\label{as:PFP}
$\lim_{m\to \infty}m^{-1}{\bar{C}_{21}}{}^{\top}\bar{C}_{21}$ exists, and positive definite. 
In addition,
$\lim_{m\to \infty}m^{-1}\bar{F}=:\bar{S}$ exists, and ${\rm rank}({P}^{\top}\bar{S}P)=n-j$. 
\end{assumption}

In this assumption, $\bar{F}$ 
corresponds to $\bar{C}^{\top}\bar{C}$ in
Assumption~\ref{as:CTC} for the TLS problem.
From this point of view,
Assumption~\ref{as:PFP}
is an extension of Assumption~\ref{as:CTC}.
More precisely,
if the conditions on $\bar{C}_{21}$ and $\bar{C}_{\rm upper}$ are excluded from Assumption~\ref{as:PFP},
then $P$ should be $I$, leading to
${\rm rank}({P}^{\top}\bar{S}P)=n$ in the same way as in Assumption~\ref{as:CTC}.

For the description of the meaning of Assumption~\ref{as:PFP},
let
\begin{align}\label{eq:Aul0}
\bar{A}_{\rm upper}=
\begin{bmatrix}
\bar{A}_{11} & \bar{A}_{12}  
\end{bmatrix},\quad
\bar{A}_{\rm lower}=
\begin{bmatrix}
\bar{A}_{21} & \bar{A}_{22}  
\end{bmatrix}.
\end{align}
The next lemma concerns the existence and uniqueness of $X$
in the asymptotic regime as $m\to \infty$.

\begin{lemma}\label{lem:assumption}
In Assumption~\ref{as:PFP}, ${\rm rank}({P}^{\top}\bar{S}P)=n-j$ is mathematically equivalent to 
\begin{align*}
{\rm rank}\left(
\begin{bmatrix}
\bar{C}_{\rm upper} \\
\bar{S}  
\end{bmatrix}\right)=n.
\end{align*}
For its upper left $(n+j)\times n$ submatrix,
\begin{align}\label{eq:Aul}
{\rm rank}\left(
\begin{bmatrix}
\bar{A}_{\rm upper} \\
\lim_{m\to \infty}m^{-1}\bar{A}_{\rm lower}{}^{\top}\bar{A}_{\rm lower}  
\end{bmatrix}
\right)=n,
\end{align}
where $\bar{A}_{\rm upper}$ and $\bar{A}_{\rm lower}$
are defined in \eqref{eq:Aul0}.
\end{lemma}

From the above lemma,
the $n$ largest singular values of the matrix in~\eqref{eq:Aul} are nonzero.
Due to the continuities of the singular values, we have the next lemma.

\begin{lemma}[Uniqueness of $X$]\label{lem:projection}
Under Assumption~\ref{as:PFP},
$X$ is uniquely determined
for all sufficiently large $m$.
\end{lemma}

\subsection{Design of the estimator with consistency analysis in the statistical sense}
In the following, we construct an estimator
of the unique $X$. To this end,
let us begin the discussion of the CTLS solution for ${C}_{\rm lower}$
because the CTLS solution has strong consistency for the regression model in \eqref{eq:crmodel}
as in Theorem~\ref{thm:Huffel}.

For the consistency analysis, noting
\begin{align*}
{{C}_{\rm lower}}^{\top}{C}_{\rm lower}=
\begin{bmatrix}
{{C}_{21}}^{\top}{C}_{21} & {{C}_{21}}^{\top}{C}_{22} \\
{{C}_{22}}^{\top}{C}_{21} & {{C}_{22}}^{\top}{C}_{22}  
\end{bmatrix},
\end{align*}
we perform the QR decomposition of $C_{21}$:
\begin{align}\label{eq:C21QR}
C_{21}=QR,\quad
Q=
\begin{bmatrix}
{Q}_{1} & {Q}_{2}
\end{bmatrix},\quad
R=
\begin{bmatrix}
{R}_{1} \\
0_{(m-j-k)\times k}
\end{bmatrix}.
\end{align}
From the above definitions, 
\begin{align}\nonumber
Q^{\top}
C_{\rm lower}
=
\begin{bmatrix}
{R}_{1} & {Q_{1}}^{\top}{C}_{22}  \\
0_{(m-j-k)\times k}   & {Q_{2}}^{\top}{C}_{22}  
\end{bmatrix}.
\end{align}
Combined with ${{C}_{\rm lower}}^{\top}{C}_{\rm lower}={{C}_{\rm lower}}^{\top}QQ^{\top}{C}_{\rm lower}$
we have
\begin{align}\nonumber
\begin{bmatrix}
{{C}_{21}}^{\top}{C}_{21} & {{C}_{21}}^{\top}{C}_{22} \\
{{C}_{22}}^{\top}{C}_{21} & {{C}_{22}}^{\top}{C}_{22}  
\end{bmatrix}
&=
\begin{bmatrix}
{{R}_{1}}^{\top}{R}_{1} & {{C}_{21}}^{\top}{C}_{22} \\
{{C}_{22}}^{\top}{C}_{21} & {{C}_{22}}^{\top}{C}_{22}  
\end{bmatrix}
\\ \label{eq:CGauss}
&=
\begin{bmatrix}
{{R}_{1}}^{\top}  \\
{{C}_{22}}^{\top}{Q_{1}}   
\end{bmatrix}
\begin{bmatrix}
{R}_{1} & {Q_{1}}^{\top}{{C}_{22}}
\end{bmatrix}
+
\begin{bmatrix}
0_{k\times k} & 0_{k\times (n+\ell-k)} \\
0_{(n+\ell-k)\times k} & {{C}_{22}}^{\top}{C}_{22}-{{C}_{22}}^{\top}{Q_{1}}{Q_{1}}^{\top}{C}_{22}  
\end{bmatrix}.
\end{align}
We focus on the last matrix on the right side of the above equation. Let
\begin{align*}
G:={{C}_{22}}^{\top}{C}_{22}-{{C}_{22}}^{\top}{Q_{1}}{Q_{1}}^{\top}{C}_{22}.
\end{align*}
Noting $C_{21}=Q_{1}R_{1}$ as in \eqref{eq:C21QR}, we have
\begin{align*}
{{C}_{22}}^{\top}{Q_{1}}{Q_{1}}^{\top}{C}_{22}
={{C}_{22}}^{\top}{C_{21}}{R_{1}}^{-1}({C_{21}}{R_{1}}^{-1})^{\top}{C}_{22}
={{C}_{22}}^{\top}{C_{21}}({C_{21}}^{\top}{C_{21}})^{-1}{C_{21}}^{\top}{C}_{22},
\end{align*}
where $R_{1}$ is nonsingular under Assumption~\ref{as:PFP}
for sufficiently large $m$.
We can therefore see that 
\begin{align}\label{eq:G}
G={{C}_{22}}^{\top}{C}_{22}-{{C}_{22}}^{\top}{C_{21}}({C_{21}}^{\top}{C_{21}})^{-1}{C_{21}}^{\top}{C}_{22}.
\end{align}
Similarly, for $\bar{F}$ in \eqref{eq:barF}, define $\bar{G}$ such that
\begin{align}\label{eq:barG}
\bar{G}={\bar{C}_{22}}{}^{\top}\bar{C}_{22}-{\bar{C}_{22}}{}^{\top}{\bar{C}_{21}}(\bar{C}_{21}{}^{\top}{\bar{C}_{21}})^{-1}\bar{C}_{21}{}^{\top}\bar{C}_{22}.
\end{align}

Let $\lambda_{1}\le \cdots \le \lambda_{\ell}$ denote
the smallest eigenvalues of $G$.
Define $\mu$ such that 
\begin{align}\label{eq:defmu}
\lambda_{1}\le \mu \le \lambda_{\ell}.
\end{align}
In addition, let
\begin{align}\label{eq:defF}
F:=
\begin{bmatrix}
{{C}_{21}}^{\top}{C}_{21} & {{C}_{21}}^{\top}{C}_{22} \\
{{C}_{22}}^{\top}{C}_{21} & {{C}_{22}}^{\top}{C}_{22}-\mu I  
\end{bmatrix}.
\end{align}

The next lemma is important for the construction of an estimator of $X$ and its asymptotic analysis.
\begin{lemma}[Strong consistency of $m^{-1}F$]\label{lem:F}
Under Assumptions~\ref{as:E} and~\ref{as:PFP},
we have 
\begin{align}\label{eq:sigmaF}
\lim_{m\to \infty}m^{-1}\mu={\sigma}^2 \ \text{and} \
\lim_{m\to \infty}m^{-1}F=\bar{S} \ \text{with probability one},
\end{align}
where $F$ is defined as in \eqref{eq:defF}.
\end{lemma}
\begin{proof}
Noting that $F$ is defined as in~\eqref{eq:defF},
we consider 
\begin{align}\label{eq:ClTCl}
\lim_{m\to \infty} m^{-1}{{C}_{\rm lower}}^{\top}{C}_{\rm lower}=
\lim_{m\to \infty} 
\begin{bmatrix}
m^{-1}{{C}_{21}}^{\top}{C}_{21} & m^{-1}{{C}_{21}}^{\top}{C}_{22} \\
m^{-1}{{C}_{22}}^{\top}{C}_{21} & m^{-1}{{C}_{22}}^{\top}{C}_{22}  
\end{bmatrix}.
\end{align}
For each submatrix, noting~\eqref{eq:CE}, we have
\begin{align}\label{eq:ClTCl2}
&
m^{-1}{{C}_{21}}^{\top}{C}_{21}=m^{-1}{\bar{C}_{21}}{}^{\top}\bar{C}_{21},
\quad
m^{-1}{{C}_{21}}^{\top}{C}_{22}=
m^{-1}{\bar{C}_{21}}^{\top}(\bar{C}_{22}+E)
\\
\label{eq:ClTCl3}
&
m^{-1}{{C}_{22}}^{\top}{C}_{21}=
m^{-1}({\bar{C}_{22}}^{\top}+E)\bar{C}_{21},
\quad
m^{-1}{{C}_{22}}^{\top}{C}_{22}=
m^{-1}({\bar{C}_{22}}+E)^{\top}(\bar{C}_{22}+E).
\end{align}
In the above four submatrices, the convergence of
$m^{-1}{{C}_{21}}^{\top}{C}_{21}$ is obvious.
For the other submatrices, 
the convergences can be proved as follows.

First, under Assumption~\ref{as:E},
the strong law of large numbers (SLLN) yields
\begin{align}\label{eq:ETEsigma2}
\lim_{m\to \infty}m^{-1}E^{\top}E=\sigma^2 I \ \text{with probability one}.
\end{align}

Next, we use Lemma~\ref{lem:kolmogolov} to
prove the convergence of $m^{-1}{\bar{C}_{21}}{}^{\top}E$
and $m^{-1}{\bar{C}_{22}}{}^{\top}E$.
Below we consider $m^{-1}{\bar{C}_{21}}{}^{\top}E$
because $m^{-1}{\bar{C}_{22}}{}^{\top}E$ can be analyzed in the same way. 
The columns of $\bar{C}_{21}$
correspond to $\alpha_{i}\ (i=1,\ldots)$,
and the columns of $E$ correspond to $\epsilon^{(i)}\ (i=1,\ldots )$.
Note that the sum of squares of the columns of $\bar{C}_{21}$
is the sum of the diagonal elements of $\bar{C}_{21}{}^{\top}\bar{C}_{21}$.
It is bounded due to \eqref{eq:defF} and the existence of $\lim_{m\to \infty}m^{-1}\bar{F}$
in Assumption~\ref{as:PFP}.
Hence, under Assumptions~\ref{as:E} and~\ref{as:PFP},
Lemma~\ref{lem:kolmogolov} leads to the strong convergence of 
$m^{-1}{{C}_{21}}^{\top}E$ to $0_{k\times (n+\ell -k)}$ as $m\to \infty$.
Similarly, we see the convergence of $m^{-1}{{C}_{22}}^{\top}E$.
In other words, we have
\begin{align}\label{eq:CTE}
\lim_{m\to \infty}m^{-1}{{C}_{21}}^{\top}E=0_{k\times (n+\ell -k)},
\quad
\lim_{m\to \infty}m^{-1}{{C}_{22}}^{\top}E=0_{(n+\ell -k)\times (n+\ell -k)},
\end{align}
with probability one.
Thus, from~\eqref{eq:ClTCl}, \eqref{eq:ClTCl2}, \eqref{eq:ClTCl3}, \eqref{eq:ETEsigma2}, and \eqref{eq:CTE}, 
we have
\begin{align}\label{eq:ClTClsigma2}
\lim_{m\to \infty} m^{-1}{{C}_{\rm lower}}^{\top}{C}_{\rm lower}=
\lim_{m\to \infty} \begin{bmatrix}
m^{-1}{{C}_{21}}^{\top}{C}_{21} & m^{-1}{{C}_{21}}^{\top}{C}_{22} \\
m^{-1}{{C}_{22}}^{\top}{C}_{21} & m^{-1}{{C}_{22}}^{\top}{C}_{22}  
\end{bmatrix}
=
\lim_{m\to \infty} \begin{bmatrix}
m^{-1}{\bar{C}_{21}}^{\top}\bar{C}_{21} & m^{-1}{\bar{C}_{21}}^{\top}\bar{C}_{22} \\
m^{-1}{\bar{C}_{22}}^{\top}\bar{C}_{21} & m^{-1}{\bar{C}_{22}}^{\top}\bar{C}_{22}+\sigma^{2}I  
\end{bmatrix}
\end{align}
with probability one.

In the following, we consider the convergence of
$m^{-1}G$ defined as in \eqref{eq:G}. Obviously, from~\eqref{eq:CGauss},
\begin{align*}
{\rm rank}(m^{-1}G)
=
{\rm rank}(m^{-1}{{C}_{\rm lower}}^{\top}{C}_{\rm lower})-k
\le n-k.
\end{align*}
For the right side on the equation \eqref{eq:G}, using
\begin{align*}
m^{-1}{{C}_{22}}^{\top}{C_{21}}({C_{21}}^{\top}{C_{21}})^{-1}{C_{21}}^{\top}{C}_{22}
=
(m^{-1}{{C}_{22}}^{\top}{C_{21}})(m^{-1}{C_{21}}^{\top}{C_{21}})^{-1}(m^{-1}{C_{21}}^{\top}{C}_{22}),
\end{align*}
we have
\begin{align*}
\lim_{m\to \infty}m^{-1}G
&=\lim_{m\to \infty}m^{-1}{C_{22}}^{\top}{C_{22}}
-(\lim_{m\to \infty}m^{-1}{{C}_{22}}^{\top}{C_{21}})(\lim_{m\to \infty}m^{-1}{C_{21}}^{\top}{C_{21}})^{-1}(\lim_{m\to \infty}m^{-1}{C_{21}}^{\top}{C}_{22})
\\
&=\lim_{m\to \infty}m^{-1}{\bar{C}_{22}}^{\top}{\bar{C}_{22}}+\sigma^2 I
-(\lim_{m\to \infty}m^{-1}{\bar{C}_{22}}{}^{\top}{\bar{C}_{21}})(\lim_{m\to \infty}m^{-1}{\bar{C}_{21}}{}^{\top}{\bar{C}_{21}})^{-1}(\lim_{m\to \infty}m^{-1}{\bar{C}_{21}{}}^{\top}\bar{C}_{22})
\\
&=\lim_{m\to \infty}m^{-1}\bar{G}+\sigma^2 I
\end{align*}
with probability one from \eqref{eq:barG} and \eqref{eq:ClTClsigma2}.
Since $\lambda_{i}\  (i=1,\ldots ,\ell)$ are the smallest eigenvalues of $G$,
we have
\begin{align*}
\lim_{m\to \infty}m^{-1}\lambda_{i}=\sigma^2 \ \text{with probability one for} \ i=1,\ldots ,\ell.
\end{align*}
Noting that $\mu$ is defined as in \eqref{eq:defmu},
we have
\begin{align*}
\lim_{m\to \infty}m^{-1}\mu=\sigma^2 \ \text{with probability one}.
\end{align*}
Combined with \eqref{eq:ClTClsigma2} 
we obtain \eqref{eq:sigmaF}.
\end{proof}

\begin{remark}
The above lemma is implied in~\cite[\S~3]{Huffel1989}.
The relevance to the Cholesky QR is implicitly used in~\cite[Theorem~2]{Huffel1989}.
However, the discussion there is analysis of expected valuese
basically in the finite dimension $m$.
Although \cite[\S~3]{Huffel1991} states that the CTLS solution
has the strong consistency,
the asymptotic analysis as $m\to \infty$ is not explicitly written
in any literature as far as the author knows.

Below we explain the known result in the existing study
more specifically.
The convergence of $m^{-1}{C}_{22}{}^{\top}{C}_{22}$
is rigorously proved in \cite[Lemma~3.1]{Gleser1981}.
However, the convergence behaviors of
$m^{-1}{C}_{21}{}^{\top}{C}_{22}$
and $m^{-1}{C}_{22}{}^{\top}{C}_{21}$
are not clear
because the assumption is not adapted to the regression model in this paper.
To complete the proof of strong convergence,
we rigorously prove Lemma~\ref{lem:kolmogolov} and use it in the following discussions.
\end{remark}

In the following,
we construct an estimator of $X$
with the use of Lemma~\ref{lem:F}.
First, we solve the eigenvalue problem of $F$
by the Rayleigh-Ritz procedure
using $P$. In other words, we compute
\begin{align}\nonumber
\widetilde{F}=P^{\top}FP
\in \R^{(n+\ell -j)\times (n+\ell -j)},
\end{align}
and then solve the eigenvalue problem:
\begin{align}\nonumber
\widetilde{F}\widetilde{\vec{z}}_{i}=\widetilde{\lambda}_{i}\widetilde{\vec{z}}_{i}\ (i=1,\ldots ,n+\ell-j),\quad
(0\le )\widetilde{\lambda}_{1}\le \cdots \le \widetilde{\lambda}_{n+\ell-j}.
\end{align}
Next, letting
\begin{align}\nonumber
\widetilde{\Lambda}_{\ell}={\rm diag}(\widetilde{\lambda}_{1}, \ldots , \widetilde{\lambda}_{\ell}),\quad
\widetilde{Z}_{\ell}=[\widetilde{\vec{z}}_{1},\ldots ,\widetilde{\vec{z}}_{\ell}],
\end{align}
we compute $Z_{\ell}$ defined as 
\begin{align}\nonumber
Z_{\ell}=P\widetilde{Z}_{\ell}\in \R^{(n+\ell)\times \ell}.
\end{align}
Let $Z_{\ell}$ be divided into
\begin{align}\nonumber
Z_{\ell}=:
\begin{bmatrix}
Z_{\ell,{\rm upper}}
\\
Z_{\ell,{\rm lower}}
\end{bmatrix},\ 
Z_{\ell,{\rm upper}}\in \R^{n\times \ell},\
Z_{\ell,{\rm lower}}\in \R^{\ell \times \ell}.
\end{align}
Then an estimator $\widehat{X}$ is given by
\begin{align}\label{eq:hatX}
\widehat{X}=-Z_{\ell,{\rm upper}}{Z_{\ell,{\rm lower}}}^{-1}.
\end{align}

The estimator $\widehat{X}$ constructed in this way has strong consistency, as in the following theorem.

\begin{theorem}[Strong consistency]\label{thm:main}
Under Assumptions~\ref{as:E} and~\ref{as:PFP},
we have
\begin{align}\label{eq:wXtoX}
\lim_{m\to \infty}\widehat{X}={X} \ \text{with probability one},
\end{align}
where $\widehat{X}$ is given by~\eqref{eq:hatX}.
\end{theorem}
\begin{proof}
Let $\bar{Z}_{\ell}\in \R^{(n+\ell)\times \ell}$ denote the matrix
comprising the Ritz vectors of $\bar{S}$
with the use of the subspace spanned by the columns of $P$.
From \eqref{eq:defP} and Lemma~\ref{lem:assumption}, we have
\begin{align*}
{\rm rank}\left(
\begin{bmatrix}
\bar{C}_{\rm upper} \\
\bar{S}  
\end{bmatrix}\right)=n,\quad
\begin{bmatrix}
\bar{C}_{\rm upper} \\
\bar{S}  
\end{bmatrix}
\bar{Z}_{\ell}=0_{(n+\ell +j)\times \ell}.
\end{align*}
Hence, letting
\begin{align}\nonumber
\bar{Z}_{\ell}=:
\begin{bmatrix}
\bar{Z}_{\ell,{\rm upper}}
\\
\bar{Z}_{\ell,{\rm lower}}
\end{bmatrix},\ 
\bar{Z}_{\ell,{\rm upper}}\in \R^{n\times \ell},\
\bar{Z}_{\ell,{\rm lower}}\in \R^{\ell \times \ell},
\end{align}
we have
\begin{align*}
\begin{bmatrix}
\bar{C}_{\rm upper} \\
\bar{S}  
\end{bmatrix}
\begin{bmatrix}
\bar{Z}_{\ell,{\rm upper}}\bar{Z}_{\ell,{\rm lower}}{}^{-1} \\
I
\end{bmatrix}
=0_{(n+\ell +j)\times \ell}.
\end{align*}
Thus, we obtain
\begin{align}\label{eq:XwithZ}
{X}=-\bar{Z}_{\ell,{\rm upper}}{\bar{Z}_{\ell,{\rm lower}}}{}^{-1}.
\end{align}
Note that the uniqueness and existence of $X$ as in Lemma~\ref{lem:projection}
imply that $\bar{Z}_{\ell,{\rm lower}}$ is nonsingular.
Recall that the estimator $\widehat{X}$ in \eqref{eq:hatX} is
computed by $Z_{\ell}$ comprising the Ritz vectors of $F$.
Lemma~\ref{lem:F} and the continuities of eigenvectors
lead to \eqref{eq:wXtoX}.
\end{proof}

It is worth noting that we have constructed a consistent estimator $\widehat{X}$ based on the properties of orthogonal projections, which naturally proves its strong consistency. With this procedure focused on orthogonal projections in mind, the next section gives a natural proof of strong consistency for the existing TLS estimator.

\section{Consistency analysis for the CTLS solution}\label{sec:mainresult2}
In this section, we prove strong consistency of $X_{\rm ctls}$ in~\eqref{eq:ctlscr}
for the regression model as in~\eqref{eq:defbarAXbarB}, with emphasis on orthogonal projections.
To this end, we start the discussion with the following simple case.

For $\bar{A}$ in~\eqref{eq:defbarAXbarB}, if $\bar{A}_{11}$ is square and nonsingular, we have
\begin{align}\label{eq:originmodel}
\begin{bmatrix}
\bar{A}_{11} & \bar{A}_{12} \\
0_{(m-j)\times k} & \bar{A}_{22}-\bar{A}_{21}{\bar{A}_{11}}{}^{-1}\bar{A}_{12} \\
\end{bmatrix}{X}=
\begin{bmatrix}
{\bar{B}}_{1} \\
{\bar{B}}_{2}-\bar{A}_{21}{\bar{A}_{11}}{}^{-1}{\bar{B}}_{1} \\
\end{bmatrix}
\end{align}
in the same manner as \eqref{eq:ctlsreduced}.
Hence,
\begin{align*}
&
\begin{bmatrix}
{A}_{11} & {A}_{12} \\
0_{(m-j)\times k} & {A}_{22}-{A}_{21}{{A}_{11}}^{-1}{A}_{12} \\
\end{bmatrix}
=
\begin{bmatrix}
\bar{A}_{11} & \bar{A}_{12} \\
0_{(m-j)\times k} & \bar{A}_{22}-\bar{A}_{21}{\bar{A}_{11}}{}^{-1}\bar{A}_{12} +E_{A_{22}} \\
\end{bmatrix},\\
&
\begin{bmatrix}
{{B}}_{1} \\
{{B}}_{2}-{A}_{21}{{A}_{11}}^{-1}{{B}}_{1}
\end{bmatrix}
=
\begin{bmatrix}
{\bar{B}}_{1} \\
{\bar{B}}_{2}-\bar{A}_{21}{\bar{A}_{11}}{}^{-1}{\bar{B}}_{1}+E_{B_{2}} \\
\end{bmatrix}.
\end{align*}
Thus the problem is reduced to the following regression model.
Letting $\widetilde{X}$ denote the lower $(n-k)\times \ell$ submatrix of $X$, we have
\begin{align}\nonumber
(\bar{A}_{22}-\bar{A}_{21}{\bar{A}_{11}}{}^{-1}\bar{A}_{12})\widetilde{X}
=
{\bar{B}}_{2}-\bar{A}_{21}{\bar{A}_{11}}{}^{-1}{\bar{B}}_{1} 
\end{align}
from~\eqref{eq:originmodel},
where $E_{A_{22}}$ and $E_{B_{2}}$ are added to
$\bar{A}_{22}-\bar{A}_{21}{\bar{A}_{11}}{}^{-1}\bar{A}_{12}$
and
${\bar{B}}_{2}-\bar{A}_{21}{\bar{A}_{11}}{}^{-1}{\bar{B}}_{1}$,
respectively.
This feature is similar to $X_{\rm ctls}$ of~\eqref{eq:ctlscr} in Section~\ref{sec:CTLSrc}.

As in Section~\ref{sec:CTLSrc},
using the SVD of $A_{11}$, 
we have \eqref{eq:Xctlsprob} without loss of generality.
Analogously, we assume $A_{11}=0_{j\times k}$, resulting in
the following regression model:
\begin{align}\nonumber
\begin{bmatrix}
0_{j\times k} & \bar{A}_{12} \\
\bar{A}_{21} & \bar{A}_{22} \\
\end{bmatrix} {X}=
\begin{bmatrix}
{\bar{B}}_{1} \\
{\bar{B}}_{2} \\
\end{bmatrix}.
\end{align}
From the definition of $\bar{C}$ in \eqref{eq:barC2} and \eqref{eq:defbarC},
\begin{align}\label{eq:barC0}
& \bar{C}=
\begin{bmatrix}
0_{j\times k} & \bar{C}_{12} \\
\bar{C}_{21} & \bar{C}_{22} \\
\end{bmatrix}\in \R^{m\times (n+\ell)}.
\end{align}

Similarly to the previous sections, 
we assume that
$\bar{C}_{12}$ is of full row rank.
If not, we select independent rows in $\bar{C}_{12}$ to recover $X$.
Thus, without loss of generality,
$\bar{C}_{12}$ is of full row rank, i.e., 
${\rm rank}(\bar{C}_{12})=j$,
leading to $j\le n-k$.
In addition, we define $\widetilde{P}$ such that
\begin{align*}
\widetilde{P}\in \R^{(n+\ell -k)\times (n+\ell -k-j)},\quad
\widetilde{P}^{\top}\widetilde{P}=I,\quad \bar{C}_{12}\widetilde{P}=0_{j\times (n+\ell -k-j)}.
\end{align*}
Moreover, we assume the following.
\begin{assumption}\label{as:PGP}
$\lim_{m\to \infty}m^{-1}{\bar{C}_{21}}{}^{\top}[\bar{C}_{21} \ \ \bar{C}_{22}]$ exists, 
and $\lim_{m\to \infty}m^{-1}{\bar{C}_{21}}{}^{\top}\bar{C}_{21}$ is positive definite. 
In addition,
$\lim_{m\to \infty}m^{-1}{\widetilde{P}}^{\top}\bar{G}\widetilde{P}=:\bar{T}$ exists, and ${\rm rank}(\bar{T})=n-j-k$. 
\end{assumption}

For the consistency analysis in this section,
the next lemma is crucial,
corresponding to Lemma~\ref{lem:F} in the previous section,
with emphasis on orthogonal projection of $\widetilde{P}$.
Note that $G$ and $\bar{G}$ can be defined as in \eqref{eq:G} and \eqref{eq:barG}, respectively,
where $\bar{C}$ is defined in this section as in \eqref{eq:barC0}.

\begin{lemma}\label{lem:PDP}
Under Assumptions~\ref{as:E} and \ref{as:PGP},
we have 
\begin{align}\label{eq:limPGP}
\lim_{m\to \infty}m^{-1}\widetilde{P}^{\top}G\widetilde{P}=\bar{T}+\sigma^2 I \ \text{with probability one}.
\end{align}
\end{lemma}
\begin{proof}
First, noting $\bar{G}$ in \eqref{eq:barG}, we have
\begin{align}\label{eq:PbarGP}
m^{-1}\widetilde{P}^{\top}\bar{G}\widetilde{P}=
m^{-1}\widetilde{P}^{\top}{\bar{C}_{22}}{}^{\top}\bar{C}_{22}\widetilde{P}
-m^{-1}\widetilde{P}^{\top}{\bar{C}_{22}}{}^{\top}{\bar{C}_{21}}(\bar{C}_{21}{}^{\top}{\bar{C}_{21}})^{-1}\bar{C}_{21}{}^{\top}\bar{C}_{22}\widetilde{P}.
\end{align}
The second term on the right side of the equaiton above is
\begin{align*}
(m^{-1}\widetilde{P}^{\top}{\bar{C}_{22}}{}^{\top}{\bar{C}_{21}})(m^{-1}\bar{C}_{21}{}^{\top}{\bar{C}_{21}})^{-1}(m^{-1}\bar{C}_{21}{}^{\top}\bar{C}_{22}\widetilde{P}).
\end{align*}
Since $\lim_{m\to \infty}m^{-1}{\bar{C}_{21}}{}^{\top}[\bar{C}_{21} \ \ \bar{C}_{22}]$ exists in Assumption~\ref{as:PGP},
\begin{align}\label{eq:PCC}
\lim_{m\to \infty}m^{-1}\widetilde{P}^{\top}{\bar{C}_{22}}{}^{\top}{\bar{C}_{21}}(\bar{C}_{21}{}^{\top}{\bar{C}_{21}})^{-1}\bar{C}_{21}{}^{\top}\bar{C}_{22}\widetilde{P}
\end{align}
exists.
In addition, 
$\lim_{m\to \infty}m^{-1}{\widetilde{P}}^{\top}\bar{G}\widetilde{P}=:\bar{T}$ exists
in Assumption~\ref{as:PGP},
and thus
\begin{align}\label{eq:PCTCP}
\lim_{m\to \infty}m^{-1}\widetilde{P}^{\top}{\bar{C}_{22}}{}^{\top}\bar{C}_{22}\widetilde{P} 
\end{align}
exists from \eqref{eq:PbarGP}.

Next, regarding ${G}$ in \eqref{eq:G},
\begin{align}
\nonumber
m^{-1}\widetilde{P}^{\top}{G}\widetilde{P} &=
m^{-1}\widetilde{P}^{\top}{{C}_{22}}{}^{\top}{C}_{22}\widetilde{P}
-m^{-1}\widetilde{P}^{\top}{{C}_{22}}{}^{\top}{{C}_{21}}({C}_{21}{}^{\top}{{C}_{21}})^{-1}{C}_{21}{}^{\top}{C}_{22}\widetilde{P}
\\
&=
m^{-1}\widetilde{P}^{\top}({\bar{C}_{22}}+E){}^{\top}(\bar{C}_{22}+E)\widetilde{P}
-m^{-1}\widetilde{P}^{\top}({\bar{C}_{22}}+E){}^{\top}{{C}_{21}}({C}_{21}{}^{\top}{{C}_{21}})^{-1}{C}_{21}{}^{\top}({\bar{C}_{22}}+E)\widetilde{P}.
\label{eq:PGP}
\end{align}
For the second term on the right side of the equation above, we have
\begin{align}\label{eq:EC21}
\lim_{m\to \infty}m^{-1}E^{\top}\bar{C}_{21}=0_{(n+\ell -k)\times k} \ \text{with probability one}
\end{align}
from Lemma~\ref{lem:kolmogolov} and the convergence of $m^{-1}\bar{C}_{21}{}^{\top}\bar{C}_{21}$
in Assumption~\ref{as:PGP}.
Thus, 
\begin{align}
&
\lim_{m\to \infty}m^{-1}\widetilde{P}^{\top}({\bar{C}_{22}}+E){}^{\top}{{C}_{21}}({C}_{21}{}^{\top}{{C}_{21}})^{-1}{C}_{21}{}^{\top}({\bar{C}_{22}}+E)\widetilde{P}
\nonumber \\
&=
\lim_{m\to \infty}(m^{-1}\widetilde{P}^{\top}({\bar{C}_{22}}+E){}^{\top}{{C}_{21}})(m^{-1}{C}_{21}{}^{\top}{{C}_{21}})^{-1})
(m^{-1}{C}_{21}{}^{\top}({\bar{C}_{22}}+E)\widetilde{P})
\nonumber \\
&=
\lim_{m\to \infty}(m^{-1}\widetilde{P}^{\top}{\bar{C}_{22}}{}^{\top}{{C}_{21}})(m^{-1}{C}_{21}{}^{\top}{{C}_{21}})^{-1}
(m^{-1}{C}_{21}{}^{\top}{\bar{C}_{22}}\widetilde{P})
\nonumber \\
&=
\lim_{m\to \infty}m^{-1}\widetilde{P}^{\top}{\bar{C}_{22}}{}^{\top}{C}_{21}({C}_{21}{}^{\top}{{C}_{21}})^{-1}
{C}_{21}{}^{\top}{\bar{C}_{22}}\widetilde{P}.
\label{eq:PGP2}
\end{align}

For the first term on the right side of the equation \eqref{eq:PGP},
\begin{align}\label{eq:PGP3}
m^{-1}\widetilde{P}^{\top}({\bar{C}_{22}}+E){}^{\top}(\bar{C}_{22}+E)\widetilde{P}
=
m^{-1}\widetilde{P}^{\top}{\bar{C}_{22}}{}^{\top}\bar{C}_{22}\widetilde{P}
+
m^{-1}\widetilde{P}^{\top}{\bar{C}_{22}}{}^{\top}E\widetilde{P}
+
m^{-1}\widetilde{P}^{\top}E{}^{\top}\bar{C}_{22}\widetilde{P}
+
m^{-1}\widetilde{P}^{\top}E{}^{\top}E\widetilde{P}.
\end{align}
Below we estimate each term on the right side of the above equation.

The strong law of large numbers yields
\begin{align}\label{eq:PETEP}
\lim_{m\to \infty}m^{-1}\widetilde{P}^{\top}E{}^{\top}E\widetilde{P}
=
\widetilde{P}^{\top}(\lim_{m\to \infty}m^{-1}E{}^{\top}E)\widetilde{P}
=
\sigma^2 I \ \text{with probability one}.
\end{align}
In addition, Lemma~\ref{lem:kolmogolov}
leads to
\begin{align*}
\lim_{m\to \infty}m^{-1}E{}^{\top}(\bar{C}_{22}\widetilde{P})=0_{(n+\ell -k)\times (n+\ell -k-j)}
\ \text{with probability one}
\end{align*}
from the diagonal elements of the matrix in \eqref{eq:PCTCP}.
Noting $m^{-1}\widetilde{P}^{\top}E{}^{\top}\bar{C}_{22}\widetilde{P}=
\widetilde{P}^{\top}m^{-1}E{}^{\top}(\bar{C}_{22}\widetilde{P})$,
we have
\begin{align}\label{eq:PECP}
\lim_{m\to \infty}m^{-1}\widetilde{P}^{\top}E{}^{\top}\bar{C}_{22}\widetilde{P}
=0_{(n+\ell -k-j)\times (n+\ell -k-j)} \ \text{with probability one}
\end{align}
Thus, \eqref{eq:PGP3}, \eqref{eq:PETEP} and \eqref{eq:PECP} yield
\begin{align}\label{eq:PGP1}
\lim_{m\to \infty}m^{-1}\widetilde{P}^{\top}({\bar{C}_{22}}+E){}^{\top}(\bar{C}_{22}+E)\widetilde{P}
=
\lim_{m\to \infty}m^{-1}\widetilde{P}^{\top}{\bar{C}_{22}}{}^{\top}\bar{C}_{22}\widetilde{P}
+\sigma^2 I,
\end{align}
where $\lim_{m\to \infty}m^{-1}\widetilde{P}^{\top}{\bar{C}_{22}}{}^{\top}\bar{C}_{22}\widetilde{P}$
exists as in \eqref{eq:PCTCP}.
From \eqref{eq:PbarGP}, \eqref{eq:PGP}, \eqref{eq:PGP2}, and \eqref{eq:PGP1},  
we obtain \eqref{eq:limPGP} under Assumptions~\ref{as:E} and \ref{as:PGP}.
\end{proof}

In the following, we prove the strong consistency
of $X_{\rm ctls}$. To this end,
let $\widetilde{\vec{z}}_{1},\ldots , \widetilde{\vec{z}}_{\ell}$
denote the Ritz vectors of $G$ corresponding to the $\ell$ smallest Ritz values,
where the subspace is spanned by the column vectors of $\widetilde{P}$.
In addition, let
$\widetilde{Z}:=[\widetilde{\vec{z}}_{1},\ldots , \widetilde{\vec{z}}_{\ell}]$.
Moreover, let $\widetilde{X}_{\rm ctls}$ denote the lower $(n-k)\times \ell$ submatrix of ${X}_{\rm ctls}$.
Then, analogously to \eqref{eq:XwithZ}, $\widetilde{Z}$ and $\widetilde{X}_{\rm ctls}$
have the following relation
\begin{align}\label{eq:ZtowXctls}
\widetilde{Z}=
\begin{bmatrix}
\widetilde{Z}_{\rm upper} \\
\widetilde{Z}_{\rm lower} 
\end{bmatrix},\quad
\widetilde{X}_{\rm ctls}=-\widetilde{Z}_{\rm upper}{\widetilde{Z}_{\rm lower}}{}^{-1}
\end{align}
from~\cite[\S~2]{Demmel1987}.
The following is a more detailed explanation.

Define $\widetilde{P}_{0}\in \R^{(n+\ell -k)\times j}$ such that 
its columns vectors comprise an orthonormal basis vectors
of the range of ${C}_{12}{}^{\top}(=\bar{C}_{12}{}^{\top})$.
Then $[\widetilde{P}_{0} \ \ \widetilde{P}]$ is an orthogonal matrix.
In addition, note the orthogonal matrix $Q$ in \eqref{eq:C21QR}.
Since we have the following orthogonal transformation of $C$:
\begin{align*}
\begin{bmatrix}
I & 0_{j\times (m-j)} \\
0_{(m-j)\times j} & Q^{\top} 
\end{bmatrix}
C
\begin{bmatrix}
I & 0_{k\times j} & 0_{k\times (n+\ell -k-j)} \\
0_{(n+\ell -k)\times k} & \widetilde{P}_{0} & \widetilde{P} 
\end{bmatrix}
=
\begin{bmatrix}
0_{j\times k} & C_{12}\widetilde{P}_{0} &  0_{j\times (n+\ell -k-j)} \\
Q_{1}{}^{\top}C_{21} & Q_{1}{}^{\top}C_{22}\widetilde{P}_{0}  & Q_{1}{}^{\top}C_{22}\widetilde{P} \\
0_{(m-j-k)\times k} &  Q_{2}{}^{\top}C_{22}\widetilde{P}_{0} & Q_{2}{}^{\top}C_{22}\widetilde{P}
\end{bmatrix},
\end{align*}
the low rank approximation of $Q_{2}{}^{\top}C_{22}\widetilde{P}$,
i.e., the lower right submatrix,
leads to the CTLS solution~\cite[Theorem~3]{Demmel1987}.

Since we can compute $\widetilde{X}_{\rm ctls}$,
the upper $k\times \ell$ part of $X_{\rm ctls}$ can be obtained by
\begin{align}\label{eq:C21X}
\begin{bmatrix}
C_{21}{}^{\top}C_{21} &  C_{21}{}^{\top}C_{22}
\end{bmatrix}
\begin{bmatrix}
X_{\rm ctls} \\
-I
\end{bmatrix}
=0_{(n+\ell)\times \ell}.
\end{align}
See~\cite[\S~2]{Demmel1987} for the details of the computation of the CTLS solution.
In \eqref{eq:C21X}, for the regularity of $C_{21}{}^{\top}C_{21}$,
$\lim_{m\to \infty}m^{-1}C_{21}{}^{\top}C_{21}$ is positive definite 
in the asymptotic regime under Assumption~\ref{as:PGP}.

\begin{theorem}[Strong consistency]\label{thm:main2}
Under Assumptions~\ref{as:E} and~\ref{as:PGP},
we have
\begin{align}\label{eq:XctlsToX}
\lim_{m\to \infty}{X}_{\rm ctls}={X} \ \text{with probability one},
\end{align}
where ${X}_{\rm ctls}$ is the solution vector of 
the constrained total least squares problem.
\end{theorem}
\begin{proof}
Let $\widetilde{X}$ denote the lower $(n-k)\times \ell$ submatrix of $X$.
Then $\widetilde{X}$ have the relation to the Ritz vectors of $\bar{G}$
in the same manner as $\widetilde{X}_{\rm ctls}$ in \eqref{eq:ZtowXctls}.
Note that, from the convergence feature as in Lemma~\ref{lem:PDP}, as $m\to \infty$,
the accumulation points of the Ritz vectors of $\bar{G}$ are those of $G$ with probability one.
Thus, we have
\begin{align*}
\lim_{m\to \infty}\widetilde{X}_{\rm ctls}=\widetilde{X} \ \text{with probability one}.
\end{align*}

Finally, we prove convergence of the upper $k\times \ell$ matrix of ${X}_{\rm ctls}$ below.
For \eqref{eq:C21X},
recall $C_{21}=\bar{C}_{21}$ from the definition.
In addition, from \eqref{eq:EC21}, we have
\begin{align*}
\lim_{m\to \infty}m^{-1}C_{21}{}^{\top}C_{22}=
\lim_{m\to \infty}m^{-1}\bar{C}_{21}{}^{\top}(\bar{C}_{22}+E)=\lim_{m\to \infty}m^{-1}\bar{C}_{21}{}^{\top}\bar{C}_{22}
\end{align*}
with probability one.
The remaining variables in the upper $k\times \ell$ matrix of ${X}_{\rm ctls}$
are computed by the linear equations with \eqref{eq:C21X},
where the corresponding coefficient matrix $m^{-1}[C_{21}{}^{\top}C_{21} \quad  C_{21}{}^{\top}C_{22}]$
converges to $m^{-1}[\bar{C}_{21}{}^{\top}\bar{C}_{21} \quad  \bar{C}_{21}{}^{\top}\bar{C}_{22}]$. 
Therefore, we obtain \eqref{eq:XctlsToX}.
\end{proof}

If we restrict the case $k=0$,
the discussion in this section can be regarded as
an extension of \cite{Aishima2022} in case of $\ell=1$
to the general $\ell \ge 1$.
In other words, the convergence proof of 
Theorem~\ref{thm:main2} covers the case $k=0$ corresponding to 
the exclusion of the condition on $\bar{C}_{21}$. 
Technically, the proof for the general $\ell \ge 1$, restricted to $k=0$, 
is almost the same as in \cite{Aishima2022}.
Importantly, $k$ is generalized.

\section{Conclusion}\label{sec:conclusion}
In this paper, we have constructed two estimators for the regression model as in~\eqref{eq:defbarAXbarB}.
Importantly, both estimators have strong consistency under reasonable assumptions.
One is an estimator $\widehat{X}$ as in~\eqref{eq:hatX}, the strong consistency is theoretically guaranteed in Theorem~\ref{thm:main}. 
The other estimator is the TLS estimator $X_{\rm ctls}$ of~\eqref{eq:ctlscr} in Section~\ref{sec:CTLSrc} with row and column constraints, 
where its strong consistency is proved in Theorem~\ref{thm:main2}.

Both estimators are derived from orthogonal projections associated with the row constraints. The key lemmas are Lemmas~\ref{lem:F} and~\ref{lem:PDP}, which state the strong convergence of the matrices containing the model parameters. It is worth noting that the above key lemmas provide the consistent estimators straightforwardly with the natural proof of strong consistency. 
More specifically, Lemma~\ref{lem:F} implies that the Rayleigh-Ritz procedure
computes the exact eigenpairs in the asymptotic regime.
Lemma~\ref{lem:PDP} states that the orthogonal projection by $\widetilde{P}$
preserves an essential structure of the matrix for constructing the estimator. 
The above procedure for designing orthogonal projection estimators is the main contribution of this paper.

\

\noindent
{\bf Acknowledgment}

\noindent
This study was supported by JSPS KAKENHI Grant Nos. JP17K14143, JP21K11909, and JP22K03422.





\end{document}